\documentclass[12pt]{amsart}
\usepackage[usenames,dvipsnames]{xcolor}
\usepackage{fancyhdr}
\usepackage{txfonts}
\usepackage{graphicx}
\usepackage{epsfig}
\usepackage{mathrsfs}
\usepackage{amssymb}
\usepackage{latexsym}
\usepackage{amsmath}
\usepackage{amsfonts}
\usepackage{amsbsy}
\usepackage{amscd}
\usepackage{amsthm}
\usepackage{subfigure}
\usepackage{verbatim}
\usepackage{color,tikz}
\usetikzlibrary{calc}

\usepackage{amsfonts}
\usepackage{amsbsy,calc}
 \usepackage{graphicx,ifthen,cite}

\newcommand{\ba}{{\mathbf{a}}}
\newcommand{\bb}{{\mathbf{b}}}
\newcommand{\bc}{{\mathbf{c}}}

\newcommand{\bi}{{\mathbf{i}}}

\newcommand{\bp}{{\mathbf{p}}}

\newcommand{\bu}{{\mathbf{u}}}

\newcommand{\bw}{{\mathbf{w}}}
\newcommand{\bx}{{\mathbf{x}}}
\newcommand{\by}{{\mathbf{y}}}
\newcommand{\bz}{{\mathbf{z}}}

\newcommand{\SD}{{\mathcal D}}
\newcommand{\SE}{{\mathcal E}}

\def\bu{{\bf u}}
\def\bw{{\bf w}}

\newtheorem{theorem}{Theorem}[section]
\newtheorem{thm}[theorem]{Theorem}
\newtheorem{prop}{Proposition}[section]
\newtheorem{cor}[prop]{Corollary}
\newtheorem{lemma}[prop]{Lemma}
\newtheorem{remark}[prop]{Remark}
\newtheorem{defi}[prop]{Definition}

\newtheorem{example}{Example}[section]
\numberwithin{equation}{section}

\def\cal{\mathcal }
\def\R{\mathbb R}

\def\Z{\mathbb Z}
\def\Q{\mathbb Q}

\def\mathscr{\mathcal }

\def\diag{\text{diag}}

\newif\ifdraft
\drafttrue
\ifdraft
\else
\fi
\numberwithin{equation}{section}

\begin{document}

\title[Invariance of multifractal spectrum of uniform self-affine measures and its applications]{Invariance of multifractal spectrum of uniform self-affine measures and its applications}

 \author{Hui Rao} \address{Department of Mathematics and Statistics, Central China Normal University, Wuhan, 430079, China
} \email{hrao@mail.ccnu.edu.cn
 }

\author{Ya-min Yang} \address{Institute of applied mathematics, College of Science, Huazhong Agricultural University, Wuhan,430070, China.
} \email{yangym09@mail.hzau.edu.cn
 }
\author{Yuan Zhang$^*$} \address{Department of Mathematics and Statistics, Central China Normal University, Wuhan, 430079, China
} \email{yzhang@mail.ccnu.edu.cn
 }

\date{\today}
\thanks {The work is supported by NSFS Nos. 11971195, 12071167 and 11601172.}


\thanks{* The correspondence author.}
\begin{abstract}
 We study the bi-Lipschitz classification of Bedford-McMullen carpets which are totally disconnected.
 Let $E$ be a such carpet and let $\mu_E$ be the uniform Bernoulli measure on $E$.
 We show that the multifractal spectrum     and   the
 doubling property of $\mu_E$ are both    invariant under a bi-Lipschitz map.
 Moreover, we show that if   $\mu_E$ and $\mu_F$ are doubling, then a bi-Lipschitz map between $E$ and $F$ enjoys a certain measure preserving property.
\end{abstract}
\maketitle


\section{\textbf{Introduction}}
Lipschitz classification is an important problem in geometrical measure theory and fractal geometry.
After the pioneer works of Cooper and Pignataro \cite{CP88} and Falconer and Marsh \cite{FaMa92}, there are many works on Lipschitz equivalence of self-similar sets, see \cite{DS,  RRX06, XX10,  Mattila,  RRW12, LL13, XiXi21, RZ15}.
Recall that two metric space $(X, d_X)$ and $(Y, d_Y)$ are said to be Lipschitz equivalent, denoted by $X\sim Y$, if there is a  bi-Lipschitz map $f: X\to Y$, precisely, there is a constant $C>0$ such that
$$
C^{-1}d_X(x,y) \leq d_Y(f(x), f(y))\leq C d_X(x,y), \text{ for all }x,y\in X.
$$
The goal of this paper is to study the Lipschitz classification of self-affine carpets, a topic which receives very few study, and is much harder than the setting
of self-similar sets.

Let $2\leq m<n$ be two  integers and denote by
$\diag (n,m)$ the diagonal matrix with diagonal entries $(n,m)$.
 Let ${\cal D}\subset\{0,1,\dots,n-1\}\times\{0,1,\dots,m-1\}$, which we call a \emph{digit set}.
 For $d\in \SD$, we define $S_d: \R^2\to \R^2$ by
 $
 S_d(z)=\diag(n^{-1}, m^{-1})(z+d).
 $
Then $\{S_d\}_{d\in \SD}$ is an iterated function system (IFS).
 The unique non-empty compact set $E=K(n,m,{\cal D})$
satisfying the set equation
$E=\underset{d\in{\cal D}}\bigcup  S_d(E)$
is called a \emph{Bedford-McMullen} carpet. In this paper, we shall call $E$ simply a \emph{self-affine carpet}.

Let us start with some notations. We set ${\mathcal M}_{t}$ to be the collection of totally disconnected self-affine carpets,
and ${\mathcal M}_t(n,m)$ to be the sub-collection of ${\mathcal M}_{t}$ with expanding matrix $\diag(n,m)$.
Let $\#A$ denote the cardinality of a set $A$.
For a digit set ${\mathcal D}\subset \{0,1,\dots, n-1\}\times\{0,1,\dots, m-1\}$, we define
\begin{equation}
a_j=\#\{i;~(i,j)\in {\cal D}\}, \quad 0\leq j\leq m-1,
\end{equation}
 and  call $(a_j)_{j=0}^{m-1}$
the \emph{distribution sequence} of ${\cal D}$, or of $K(n,m,{\cal D})$.
We denote
$$\sigma=\log m/\log n, \ N=\# \mathcal{D} \text{ and }s=\#\{ j\in \{0,1,\dots, m-1\};~ a_j>0 \}.$$

Clearly, the Hausdorff, box, and  Assouad dimensions are all Lipschitz invariants.
 (The first two dimensions are computed by Bedford \cite{Bed84} and Mcmullen\cite{M84}, while
 the third one is computed by J. Mackay \cite{MM11}.) A set $K$ is said to be \emph{regular}, if $\dim_H K=\dim_B K$, see Falconer \cite{Fal90};  clearly, the regularity property is   a Lipschitz invariant.
A self-affine carpet  is regular if and only if
  it has uniform horizontal fibers, that is , all non-zero $a_j$ are equal (\cite{Bed84, M84}).

 Up to now, there are two papers on the Lipschitz classification of self-affine carpets.
 Under a certain vertical separation condition,  Li, Li and Miao\cite{Miao2013}   proved that
if $E, F\in {\cal M}_t(n,m)$    share the same
distribution sequence,  then $E\sim F$.

For a digit set $\SD$, we say the $j$-th row of $\SD$ is \emph{vacant} if $a_j=0$.
   Miao, Xi and Xiong\cite{Miao2017}  showed that
if two self-affine carpets are totally disconnected and are Lipschitz equivalent,
 then  either both of them possess vacant rows or neither  of them does.

In the study of  Lipschitz classification of self-similar sets, the Hausdorff measure is a useful tool.
However, Peres \cite{PY94} proved that if a self-affine carpet is not regular, then its Hausdorff measure (in its  dimension) is always infinity. In the present paper,   we will make use
 of the uniform Bernoulli measure  instead of the Hausdorff measure.

Let $E=K(n,m,\SD)$   and let $\bp=(p_d)_{d\in \SD}$ be a probability weight.
Then there is a unique Borel probability measure $\mu_\bp$ on E satisfying
\begin{equation}\label{def_mu}
\mu_{\bp}(\cdot)=\sum_{d\in \SD} p_d \mu_\bp\circ S_d^{-1}(\cdot)
\end{equation}
 and we call $\mu_\bp$ a \emph{self-affine measure}, or a \emph{Bernoulli measure} (\cite{Hut81}).
We denote by $\mu_E$ the self-affine measure with the weight
$p_d=1/N$ for all $d\in \SD$, and call it  the \emph{uniform Bernoulli measure} of $E$.
The main concern of the present paper is to develop Lipschitz invariants related to the uniform Bernoulli measure.
In the following we describe the main results of the present paper.

For $\bi=d_1\dots d_k\in \SD^k$, we define
$S_{\bi}(z)=S_{d_1}\circ\cdots \circ S_{d_k}(z)$;
we  call  $S_{\bi}([0,1]^2)$ a  \emph{basic rectangle} of rank $k$, and call    $E_{\mathbf i}=S_{\mathbf i}(E)$ a \emph{cylinder} of rank $k$.
We show that for every cylinder $E_{\bi}$ and every $\delta>0$,  $\mu_E(E_{\bi})\delta^{-\dim_B E}$ gives a very accurate estimate of the number of $\delta$-mesh boxes intersecting $E_{\bi}$ (Theorem \ref{thm:count}). This leads to the following

\begin{theorem}\label{thm:equivalent} Let $E, F\in {\cal M}_t(n,m)$.
If $f:~E\to F$ is a bi-Lipschitz map, then $\mu_F\circ f$ is equivalent to
$\mu_E$, namely, there exists $\zeta>0$ such that
\begin{equation}\label{eq:AC}
\zeta^{-1} \mu_E(A) \leq      \mu_F(f(A))\leq \zeta \mu_E(A)
\end{equation}
for any Borel set $A\subset E$.
\end{theorem}

As a consequence of the  above theorem, we have

\begin{cor}\label{cor:spectra} If $E, F\in {\cal M}_t(n,m)$ and $E\sim F$, then
$(i)$ $\mu_E$ and $\mu_F$ have the same multifractal spectrum;
$(ii)$  $\mu_E$ is doubling  if and only if $\mu_F$ is doubling.
 \end{cor}

 \begin{remark}\label{rem:Wen}{\rm
  A   measure $\mu$ on a metric space $X$
is said to be \emph{doubling} if there is a constant $C\geq 1$ such that
$0<\mu(B(x,2r))\leq C\mu(B(x,r))<\infty$
for all balls $B(x,r)\subset X$ of radius $r$.
Li, Wei and Wen \cite{LWW16} characterized when a  Bernoulli measure  on a self-affine carpet is doubling.
According to their result, the uniform Bernoulli measure   $\mu_E$  is doubling
 if and only if  either  (i) $a_0a_{m-1}=0$, or (ii) $a_ja_{j+1}=0$ for all $0\leq j\leq m-2$, or  (iii) $a_0=a_{m-1}$.
}
\end{remark}

The multifractal spectrum of self-affine measures on self-affine carpets was first studied by
King \cite{King95}.
King obtained a formula of the multifractal spectrum under a certain separation condition.
 Barral and Mensi \cite{Bar} relaxed the condition, and
Jordan and Ram \cite{JR11} completely removed the condition.
Olsen\cite{O98} considered the multifractal analysis of the higher dimensional self-affine sponges.

Using  the spectrum formula of King,  we characterize when $\mu_E$ and $\mu_F$,
  where $E=K(n,m, \SD)$ and $F=K(n,m,\SD')$,
 have the same multifractal spectrum.
We use $(a_j)_{j=0}^{m-1}$ and $(b_j)_{j=0}^{m-1}$ to denote the distribution sequences
of $\SD$ and $\SD'$, respectively. Denote
$N'=\#\SD'$ and $s'=\#\{j;~b_j>0\}$. Let
$$a_1^*>  a_2^*> \cdots>  a_{\tilde p}^*$$
be  the distinct non-zero terms of   $(a_j)_{j=0}^{m-1}$ and  let
 $M_i$  be  the occurrence of $a_i^*$; similarly, let
$
b_1^* > b_2^*>\cdots>  b_{\tilde q}^*
$
be the  distinct non-zero terms of     $(b_j)_{j=0}^{m-1}$, and let
  $M'_i$  be the occurrence of $b_i^*$.

\begin{thm}\label{thm:invariants}
Let $E=K(n,m,\SD)$ and $F=K(n,m,\SD')$ be two self-affine carpets. Then $\mu_E$ and $\mu_F$ have the same multifractal spectrum if and only if
\begin{equation}\label{eq:invariants}
\tilde p=\tilde q \ \  \text{ and } \ \
\frac{a_i^*}{b_i^*}=\left (\frac{M_i'}{M_i}\right )^{1/\sigma}=\left(\frac{s'}{s}\right )^{1/{\sigma}}=\left(\frac{N}{N'}\right)^{1/(1-\sigma)},  ~~\text{for }~~ i=1,\dots, \tilde p.
\end{equation}
\end{thm}

\begin{remark}\label{cor:dimensions}{\rm Using the above theorem, it is easy to show  that if  $\mu_E$ and $\mu_F$ have the same multifractal spectrum, then
 $E$ and $F$ share the same Hausdorff, box, and Assouad dimensions.
 So the above multifractal spectrum is a stronger invariant than the dimensions.
 }
 \end{remark}

In the following, we confine our study to self-affine carpets which are totally disconnected, possess vacant rows and the uniform Bernoulli measures are doubling; we
   use   ${\mathcal M}_{t,v,d}(n,m)$ to  denote the set of such  carpets.
For such  carpet $E$, we show that   $\mu_E$ satisfies an `arithmetic' doubling property, that is, if two approximate squares of $E$
are not far from each other, then the ratio of their measures is a rational number with a fixed denominator
(Lemma \ref{lem:UU'}).

Measure preserving property of bi-Lipschitz maps between Cantor sets was first observed by Cooper and Pignataro \cite{CP88} and
Falconer and Marsh \cite{FaMa92}. It is extended to   general self-similar sets  by Xi and Ruan \cite{XR08},
and plays a significant r\^ole in  many works (see \cite{RRW12, RZ15}).
Thanks to the arithmetic doubling property, we show that

\begin{theorem}\label{thm:measure}   Let  $E, F\in {\cal M}_{t,v,d}(n,m)$. If $f:~E\to F$ is a bi-Lipschitz map, then
there exists a cylinder  $E_{\bi}$ such that
$f:~({E_{\bi}},\mu_E)\to (f(E_{\bi}),\mu_F)$
is  measure preserving in the sense that, for any Borel subset $B\subset E_{\bi}$,
$$
\frac{ \mu_F (f(B))}{\mu_E(B)}=\frac{\mu_F (f(E_{\bi}))}{\mu_E(E_{\bi})}.
$$
\end{theorem}

Using Theorem \ref{thm:measure}, by a number theoretical argument, we obtain an invariant   stronger than  the multifractal spectrum
  when   $\sigma=\log m/\log n$ is irrational.

\begin{theorem}\label{thm:irrational}  Let  $E, F\in {\cal M}_{t,v,d}(n,m)$
and assume that  $\log m/\log n\in \Q^c$. If $E\sim F$,  then the distribution sequence of $E$  is a  permutation of that of $F$.
\end{theorem}

\begin{remark}{\rm   Based on the results of the present paper,  in a sequential paper, Yang and Zhang \cite{YZ20a} proves the following:  Let  $E, F\in {\cal M}_{t,v}(n,m)$ and assume that both $E$ and $F$ are regular.
Then (i) If   $\log m/\log n\in \Q$,  then $E\sim F$ if and only if $\mu_E$ and $\mu_F$ have the same multifractal spectrum;  (ii)  If   $\log m/\log n\in \Q^c$, then $E\sim F$ if and only if the distribution sequence of $E$ is a permutation of that of $F$.
}
 \end{remark}

 However, for general self-affine carpets, the complete Lipschitz classification is still a tedious task.

\begin{example}\label{exam:two} {\rm Let $m=4$, $n=6$.
Two digit sets ${\cal D}$ and ${\cal D}'$ are shown in Figure \ref{fig:two}.
  Then $E=K(6,4, {\mathcal D})$ and $F=K(6,4, {\mathcal D}')$
  are not Lipschitz equivalent since $\mu_E$ is doubling but $\mu_F$ is not.
}
\end{example}

\begin{figure}[hppt]
  \begin{tikzpicture}[xscale=.6,yscale=.9]
        \pgfmathsetmacro{\h}{10}
        \foreach \i in {\h,0} \draw(\i,0)grid++(6,4);
        \foreach \i in {{0,0},{2,0},{4,0},{1,1},{2,1},{1,2}}
        \draw[fill=red!70](\i)rectangle++(1,1);
        \foreach \i in {{0,0},{2,0},{4,0},{1,1},{1,3},{2,3}}
        \draw[fill=red!70]($(\i)+(\h,0)$) rectangle++(1,1);
  \end{tikzpicture}
 \caption{\label{fig:two}The digit sets ${\cal D}$ and ${\cal D'}$ in Example \ref{exam:two}.}
\end{figure}

\begin{example}\label{exam:one} {\rm Let $m=8$, $n=27$, then $\sigma=\log 2/\log 3.$
Let $\SD$ and $\SD'$ be the digit sets illustrated by Figure \ref{fig:one}.
Then, for digit set ${\cal D}$, we have  $N=9$, $s=2$, $(a^*_1, a^*_2)=(6,3)$ and $M_1=M_2=1$.
For digit set ${\cal D}'$, we have  $N'=6$, $s'=4$, $(b^*_1, b^*_2)=(2,1)$ and $M_1'=M_2'=2$.
One can check $E=K(n,m, {\mathcal D})$ and $F=K(n,m, {\mathcal D}')$ satisfy  \eqref{eq:invariants}, and hence   $\mu_E$ and $\mu_F$ have the same multifractal spectrum. However, by Theorem \ref{thm:irrational},
$E$ and $F$ are not Lipschitz equivalent.
}
\end{example}

\begin{figure}[hppt]
    \begin{tikzpicture}[xscale=.1334,yscale=.45]
        \pgfmathsetmacro{\h}{45}
        \foreach \i in {\h,0} \draw(\i,0)grid++(27,8);
        \foreach \i in {1,4,7}
            \draw[fill=red!70](\i,1)rectangle++(1,1);
        \foreach \i in {1,...,6}
            \draw[fill=red!70](2*\i-1,4)rectangle++(1,1);
        \foreach \i in {0,...,3}
            \draw[fill=red!70](\h+1,2*\i)rectangle++(1,1);
        \foreach \i in {4,6}
            \draw[fill=red!70](\h+3,\i)rectangle++(1,1);
  \end{tikzpicture}
 \caption{The digit sets ${\cal D}$ and ${\cal D'}$ in Example \ref{exam:one}.}
 \label{fig:one}
\end{figure}

\medskip

The paper is organized as follows. In Section \ref{sec:geo}, we investigate the basic rectangles
of self-affine carpets. Theorem \ref{thm:equivalent} and Corollary \ref{cor:spectra} are proved in Section \ref{sec:invariant}. Theorem \ref{thm:invariants} is proved in Section \ref{sec:compare}. In Section \ref{sec:vacant}, we discuss approximate squares.
In Section \ref{sec:arithmetic}, we investigate the arithmetic doubling property of $\mu_E$. Theorem \ref{thm:measure} and
 Theorem \ref{thm:irrational} are proved in Section \ref{sec:preserve} and \ref{sec:irrational}, respectively.


\section{\textbf{Basic rectangles of   self-affine carpets}}\label{sec:geo}

 Let $E=K(n,m, \SD)$ be a  self-affine carpet.
Throughout the paper, we will use the  notation
$\ell(k)=\lfloor k/\sigma\rfloor$
where $\sigma=\log m/\log n$ and $\lfloor x \rfloor$ denotes the greatest integer no larger than $x$.  Recall
that $(a_j)_{j=0}^{m-1}$ is the distribution sequence of $E$; we denote
\begin{equation}\label{def:SE}
{\mathcal E}=\{j;~ a_j>0\}\ \text{ and } \ \ s=\#{\mathcal E}.
\end{equation}

 Set $\widetilde {\mathbf E}_k=\bigcup_{\bi\in \SD^k} S_{\bi}([0,1]^2)$
and we call it the \emph{$k$-th approximation} of $E$.
Let $\bi=(x_1,y_1)\dots (x_k,y_k)\in \SD^k$, we will denote
   the basic rectangle $S_{\bi}([0,1]^2)$ by $R(x_1\dots x_k, y_1\dots y_k)$.

Let $q\geq 2$ be an integer, and
$x_1\dots x_k\in \{0,1,\dots, q-1\}^k$, we will use the notation
$ {0.x_1\dots x_k}|_q=\sum_{j=1}^k x_j q^{-j}.$
For $\bx=x_1\dots x_k\in \{0,1,\dots, n-1\}^k$ and $\by=y_1\dots y_{\ell(k)}\in \{0,1,\dots, m-1\}^{\ell(k)}$, set
\begin{equation}
{Q}(\bx,\by)=({0.\bx}|_n,  {0.\by}|_m)+
 \left [0,\frac{1}{n^k}\right ]\times \left [0,\frac{1}{m^{\ell(k)}}\right ]
\end{equation}
and we call it an \emph{  approximate square} of rank $k$,
 if $(x_j, y_j)\in \SD$ for $j=1,\dots, k$ and
$y_j\in \SE$ for $j=k+1, \dots, \ell(k)$  (See \cite{M84}).

The following lemma   has been used in literature as an obvious fact.
\begin{lemma}\label{lem:boundary} Let $E=K(n,m, \SD)$ be a self-affine carpet. Then any two basic rectangles of rank $k$ are disjoint in measure $\mu$.
\end{lemma}

\begin{proof} If $\SD$ is located in a single column or in a single row, obviously
   the lemma holds. Otherwise,
    let $\pi(x,y)=x$, then the projection measure  $\nu=\mu_E\circ \pi^{-1}$
  is a continuous measure, which implies that any vertical line segment has measure $0$ in $\mu_E$;
  similarly    any horizontal line segment has measure $0$ in $\mu_E$.
 The lemma is proved.
   \end{proof}

Let $\delta>0$. We call $\delta(z+[0,1]^d)$ a \emph{$\delta$-mesh-box} when $z\in \Z^d$.
For a bounded set  $A\subset \mathbb{R}^d$, we define  $N_\delta(A)$ to be the number of $\delta$-mesh-boxes intersecting $A$.
Recall that the box dimension of $E$ is given by  (see \cite{Bed84, M84})
\begin{equation}\label{eq:box}
\dim_B E=\log_n(Ns^{1/\sigma-1}).
\end{equation}

\begin{thm}\label{thm:count}
 Let $R=R(\bx,\by)$ be a basic rectangle of rank $k$, let $p\geq 0$ be an integer,  and let $\delta=1/n^{k+p}$.
 Then $N_\delta(R\cap E)$ is comparable to $\mu_E(R)\delta^{-\dim_B E}$, that is
\begin{equation}\label{eq:C2}
C_2^{-1}\mu_E(R)\delta^{-\dim_B E}\leq  N_\delta(R\cap E)\leq C_2\mu_E(R)\delta^{-\dim_B E},
\end{equation}
where $C_2=2s(m+2)$.
\end{thm}

\begin{proof}  Notice that $R$ contains $N^p$ number of basic rectangles of rank $k+p$. For every   basic rectangle $J$ of rank $k+p$,
there are $s^{\ell(k+p)-(k+p)}$ number
of approximate squares of rank $k+p$ contained in $J$, see \cite{M84}.
Clearly $N_\delta(R\cap E)$ is comparable to $s^{\ell(k+p)-(k+p)}$; precisely,  setting $C'=2(m+2)$, we have
$$
(C')^{-1} N^{p}s^{\ell(k+p)-(k+p)}\leq N_\delta(R\cap E)\leq C' N^ps^{\ell(k+p)-(k+p)}.
$$
Since $\dim_B E=\log_n(Ns^{1/\sigma-1})$,  we have $\delta^{-\dim_B E}=N^{k+p} s^{(1/\sigma-1)(k+p)}.$
Notice that $\ell(k+p)\leq (k+p)/\sigma<\ell(k+p)+1$, we obtain \eqref{eq:C2} by setting $C_2=sC'$.
\end{proof}

\subsection{Connected components of $\widetilde {\mathbf E}_k$}
 Miao \textit{et al.} \cite{Miao2017} proved the following result.

\begin{thm}\label{thm:finite-type} (\cite{Miao2017})
  If  $E\in {\mathcal M}_t(n,m)$, then there is an integer $M_0$ such that for all $k\geq 1$,
  each connected component of  $\widetilde {\mathbf E}_k$ consists of at most $M_0$ basic rectangles of rank $k$.
\end{thm}

 The following lemma is an analogue of    Falconer and Marsh \cite[Lemma 3.2]{FaMa92}.
For a self-affine carpet   $F$,   we use  $\widetilde {\mathbf F}_k$
 to denote its $k$-th approximation.

\begin{lemma}\label{lem:Falconer}
 Let $E, F\in {\mathcal M}_t(n,m)$.
Let $f:E\to F$ be a bi-Lipschitz map.
 Let $k\geq 1$ and let $U$ be a connected component  of $\widetilde {\mathbf E}_k$.
Then there exist an integer $p=p(k)$
and  a set of  connected components of  $\widetilde {\mathbf F}_{k+p}$, which we denote by $J_j$, $1\leq j\leq q$,
such that
\begin{equation}\label{eq:Falconer}
f(U \cap E)=\bigcup_{j=1}^q (J_j\cap F).
\end{equation}
\end{lemma}

\begin{proof}
Let $M_0$ be a constant such that Theorem \ref{thm:finite-type} holds for $E$ and $F$ simultaneously.
Let $C_0$ be a  Lipschitz constant of $f$.
Let $p=p(k)$ be an integer satisfying
$\displaystyle \frac{2M_0 }{m^{k+p}}< \frac{1}{C_0n^k}.$
Let $J$ be a  connected component of $\widetilde {\mathbf F}_{k+p}$, we claim that
  $J\cap F$ is either contained in $f(U\cap E)$, or it is disjoint from  $f(U\cap E)$.

Suppose on the contrary that   there exists a connected component
$J$ of   $\widetilde {\mathbf F}_{k+p}$ such that on one hand,
there exists $x\in U\cap E$ with $f(x)\in J$,
 and on the other hand, there exists $y\in J\cap F$  such that $f^{-1}(y)\notin U\cap E$.
The fact that   $x$ and $f^{-1}(y)$  belong to different components of $\widetilde {\mathbf E}_k$ implies that
$|x-f^{-1}(y)|\geq \frac{1}{n^k}$, so
\begin{equation}\label{equ_61}
|f(x)-y|\geq \frac{1}{C_0n^k}.
\end{equation}
On the other hand, since $f(x),y\in J$, we have
\begin{equation}\label{equ_71}
|f(x)-y|\leq \text{diam}(J)\leq \frac{2M_0 }{m^{k+p}}.
\end{equation}
Relations \eqref{equ_61} and \eqref{equ_71} imply that
$
\displaystyle \frac{2M_0 }{m^{k+p}}\geq \frac{1}{C_0n^k},
$
which contradicts the choice  of $p$.
 The claim is proved, and the lemma follows.
\end{proof}

\section{\textbf{Invariance  of multifractal spectrum and doubling property}}\label{sec:invariant}

In this section we prove Theorem \ref{thm:equivalent} and Corollary \ref{cor:spectra}.
The following lemma is obvious.

\begin{lemma}\label{lem:box-counting} Let $X$ and $Y$ be two bounded sets in $\R^d$, and let
  $f:~X\to Y$ be a bi-Lipschitz map with  Lipschitz constant  $c$. Let $C_1=(2c\sqrt{d}+2)^d$.
 Then  for any $\delta>0$ we have
$$
C_1^{-1}N_\delta(X) \leq N_\delta(Y)\leq C_1 N_\delta(X).
$$
\end{lemma}

\begin{thm}\label{main_th1} Let $E, F\in {\mathcal M}_t(n,m)$, and
let $f:~E\rightarrow F$ be a  bi-Lipschitz map. Then there exists $\zeta>0$,
such that for any $k\geq 1 $ and any  connected component  $U$ of $\widetilde {\mathbf E}_k$, it holds that
\begin{equation}\label{main_re_1}
  \mu_F(f(U\cap E)) \leq \zeta\mu_E(U\cap E).
\end{equation}
\end{thm}
\begin{proof}  First, by Lemma \ref{lem:Falconer}, we have
$\displaystyle f(U\cap E)=\bigcup_{j=1}^q (J_j\cap F),$
where $J_j$'s are  connected components of $\widetilde {\mathbf F}_{k+p}$.
Set $\delta= n^{-(k+p)}$ and let $\beta$ be   the box dimension of $E$ (also $F$).
By
Lemma  \ref{lem:box-counting}, there is a constant $C_1>0$ such that
\begin{equation}\label{eq:est_6}
N_\delta(U\cap E)\geq C_1^{-1} N_\delta\left (\bigcup_{j=1}^q (J_j\cap F)\right ).
\end{equation}
Let $C_2$ be a constant such that  Theorem \ref{thm:count} holds for $E$ and $F$ simultaneously.
By  Theorem \ref{thm:count} and Lemma   \ref{lem:boundary},  we have
\begin{equation}\label{est_4}
 N_\delta(U\cap E) \leq C_2\mu_E(U)\delta^{-\beta}
\end{equation}
and, since a $\delta$-mesh box can intersect at most four basic rectangles of rank $k+p$, we have
\begin{equation}\label{est_5}
N_\delta\left (\bigcup_{j=1}^q (J_j\cap F)\right ) \geq \frac{1}{4}\sum_{j=1}^qN_\delta(J_j\cap F)\geq (4C_2)^{-1} \delta^{-\beta} \mu_F\left (\bigcup_{j=1}^q (J_j\cap F)\right ).
\end{equation}
Combining  \eqref{eq:est_6},  \eqref{est_4} and \eqref{est_5}, we obtain
$$
\mu_F \left (\bigcup_{j=1}^q (J_j\cap F)\right )  \leq 4C_1  C_2^2 \mu_E(U\cap E).
$$
The theorem is proved.
\end{proof}

\begin{proof}[\textbf{Proof of Theorem \ref{thm:equivalent}.}]
Since the Borel $\sigma$-algebra of $E$ can be generated by
$$
{\mathcal B}_0=\bigcup_{k=1}^\infty \{U\cap E;~ \text{$U$ is a connected component of $\widetilde {\mathbf E}_k$}\},
$$
 it follows that $\mu_F(f(A))\leq \zeta\mu_E(A)$ holds for all Borel set $A\subset E$.
 Changing the role of $E$ and $F$, we obtain the other side inequality.
\end{proof}

Let $\mu$ be a Borel measure on a metric space  $E$.
Let $B(x,r)$ be the ball with center $x$ and radius $r$.
For any $x\in E$, the upper and lower local dimension of $\mu$ at
$x$ are defined by
\begin{equation}
\overline{d}_\mu(x)=\underset{r\rightarrow 0}{\limsup} \frac{\log \mu(B(x,r)\cap E)}{\log r}
\ \text{ and } \
\underline{d}_\mu(x)=\underset{r\rightarrow 0}{\liminf} \frac{\log \mu(B(x,r)\cap E)}{\log r}
\end{equation}
respectively. If $\overline{d}_\mu(x)=\underline{d}_\mu(x)$, then we denote the common value
by $d_\mu(x)$, and call it the \emph{local dimension} of $\mu$ at $x$.

For $\alpha\in \R$, the level sets $X_{\alpha,E}$ are defined by
\begin{equation}
X_{\alpha,E}=\{x\in E:~~d_\mu(x)=\alpha\}.
\end{equation}
The function
$h_\mu(\alpha)=\text{dim}_H X_{\alpha,E}$
is called  the  \emph{multifractal spectrum} of $\mu$.

The following lemma is obvious, and we leave its proof to the reader.

\begin{lemma}\label{thm:equivalence} Let $E$ and $F$ be two metric spaces and let $f: E\to F$ be a bi-Lipschitz map.
Let $\mu,\nu$ be two probability measures on $E$ and $F$ respectively. If $\mu$ and $\nu\circ f$ are equivalent, then

(i) $\mu$ and $\nu$ have the same multifractal spectrum;

(ii) $\mu$ is doubling if and only if $\nu$ is doubling.
\end{lemma}

\begin{proof}[\textbf{The proof of Corollary \ref{cor:spectra}}]
It is a direct consequence of  Theorem \ref{thm:equivalent} and Lemma \ref{thm:equivalence}.
\end{proof}

\section{\textbf{Proof of Theorem \ref{thm:invariants}}}\label{sec:compare}

In this section, we characterize when $\mu_E$ and $\mu_F$ have the same multifractal spectrum.

\subsection{Multifractal spectrum of self-affine measures}
Let $E=K(n,m,\SD)$ be a self-affine carpet. Let
  $\mu_\bp$  be the self-affine measure with the weight $\bp=(p_d)_{d\in \SD}$.
  \cite{King95} and \cite{JR11} determined the multifractal spectrum of $\mu_{\bp}$.
In the following, we describe their results, but only for  the uniform Bernoulli measure $\mu_E$.

Recall that   $\SE=\{j;~~a_j>0\}$ and $s=\#\SE$.
Fix $t>0$.
Define $\beta_E(t)$ to be the unique positive solution of
\begin{equation}\label{def_beta}
m^{\beta_E(t)}{N^{-t}}\underset{j\in \SE}{\sum}    a_j^{\sigma+(1-\sigma)t}=1.
\end{equation}
Set
$$
\alpha_{\min}=\frac{\sigma-1}{\log m}~\log (\max_{j\in \SE} a_j)+\frac{\log N}{\log m},
\quad
\alpha_{\max}=\frac{\sigma-1}{\log m}~\log (\min_{j\in \SE} a_j) +\frac{\log N}{\log m}.
$$

\begin{thm}\label{thm:JR} (\cite{King95, JR11})
For any $\alpha\in(\alpha_{\min},\alpha_{\max})$, we have
\begin{equation}\label{formular_h}
h_E(\alpha)=\text{dim}_H X_{\alpha,E}=\underset{t}{\inf}(\alpha t+\beta_E(t)).
\end{equation}
In other words, $h_E$ is the Legendre transform of $\beta_E$. Furthermore $h_E$
is differentiable with respect to $\alpha$ and is concave.
\end{thm}

We remark that if $E$ is regular, then $\alpha_{\min}=\alpha_{\max}=\dim_HE$, and $h_E(\alpha_{\min})=\dim_HE$.
The  following lemma will be needed later.

\begin{lemma}\label{lem:Legendre} The function $\beta_E$ is the Legendre transform of $h_E$.
\end{lemma}

\begin{proof} It is shown that  $\beta_E$ is a concave function, see King \cite[Theorem 1]{King95}.
Under this circumstance, $\beta_E$ is the Legendre transform of $h_E$ (see Zorich \cite[Page 262]{Zorich}).
\end{proof}

\subsection{When $\mu_E$ and $\mu_F$ have the same multifractal spectrum}
Let $F=K(n,m, \SD')$ be another self-affine carpet.
Let $(b_j)_{j=0}^{m-1}$ be the distribution sequence of $\SD'$,
let $N'=\#\SD'$,  $\SE'=\{j;~b_j>0\}$ and $s'=\#\SE'$.

Similarly,  fix $t>0$  and define $\beta_F(t)$ to be the unique positive solution of
\begin{equation}\label{def:betaF}
m^{\beta_F(t)}{(N')^{-t}}\underset{j\in \SE'}{\sum}    b_j^{\sigma+(1-\sigma)t}=1.
\end{equation}

Recall that  $\{a_j:~ j\in \SE\}=\{a_1^*>a_2^*>\dots>a_{\tilde{p}}^*\}$,
$\{b_j:~ j\in \SE'\}=\{b_1^*>b_2^*>\dots >b_{\tilde{q}}^*\},$
 $M_i$ is the occurrence of $a_i^*$ in $(a_j)_{j=0}^{m-1}$
and  $M_i'$ is the occurrence of $b_i^*$ in $(b_j)_{j=0}^{m-1}$. (See Section 1.)

\begin{proof}[\textbf{Proof of Theorem \ref{thm:invariants}}]
First, we prove that   $\mu_E$ and $\mu_F$ have the same multifractal spectrum implies \eqref{eq:invariants}.
  In this case, either both $E$ and $F$ are regular or none of them is  regular, since  $\alpha_{\min}=\alpha_{\max}$ if $\mu_E$ is regular and $\alpha_{\min}<\alpha_{\max}$ otherwise.

If both $E$ and $F$ are regular, then  $\tilde{p}=\tilde{q}=1$ and $\dim_B E=\dim_B F$.
By the dimension formula \eqref{eq:box}, it is easy to show that   \eqref{eq:invariants} holds.

Now we assume that neither $E$ nor $F$ is regular.
By Lemma \ref{lem:Legendre}, $\beta_E=\beta_F$ since they are the Legendre transform of a same function.
Therefore,
\begin{equation}\label{equ_EF}
\frac{1}{N^t}\underset{j\in \SE}{\sum}  a_j^{\sigma+(1-\sigma)t}
=\frac{1}{(N')^t} \underset{j\in \SE'}{\sum}   b_j^{\sigma+(1-\sigma)t}\quad \text{for }~ t>0.
\end{equation}
 In terms of $M_j, a_j^*$ and $M_j', b_j^*$, we obtain
\begin{equation}\label{equ_EFMax}
\frac{1}{N^t}\sum_{j=1}^{\tilde p}M_j (a^*_j)^\sigma \cdot (a^*_j)^{(1-\sigma)t}
=\frac{1}{(N')^t} \sum_{j=1}^{\tilde q} M'_j (b^*_j)^\sigma \cdot (b^*_j)^{(1-\sigma)t}\quad \text{for }~ t>0.
\end{equation}
Setting $x_j= N'(a^*_j)^{1-\sigma}$ for $1\leq j \leq \tilde p$ and  $y_j=N(b^*_j)^{1-\sigma}$ for $1\leq j \leq \tilde  q$, we obtain
\begin{equation}\label{beta_EF}
\sum_{j=1}^{\tilde p}M_j (a^*_j)^\sigma \cdot x_j^t=
\sum_{j=1}^{\tilde q}M'_j (b^*_j)^\sigma \cdot y_j^t
\quad \text{for }~ t>0.
\end{equation}
We note that $x_j\geq N'$ and $y_j\geq N$.
Moreover, since $E$ and $F$ are irregular, we have $N,N'\geq 3$, and consequently  $\ln x_j, \ln y_j>1$.

Taking the $k$-th derivative of $t$ to both sides of \eqref{beta_EF}, we get
\begin{equation}\label{main_equ}
\sum_{j=1}^{\tilde p}M_j (a^*_j)^\sigma \cdot (\ln x_j)^k \cdot x_j^t=
\sum_{j=1}^{\tilde q}M'_j (b^*_j)^\sigma \cdot (\ln y_j)^k \cdot y_j^t
\quad~\text{ for }~ t>0.
\end{equation}

First, we claim that
$x_1=y_1.$
 Notice that  $x_1=N' (a_1^*)^{1-\sigma}$
is strictly larger than the other $x_j$ ,
and $y_1=N(b_1^*)^{1-\sigma}$ is strictly larger than the other $y_j$.
Fix $t>0$,
  then both sides of \eqref{main_equ}  are exponential functions of the variable   $k$,
   and $(\ln x_1)^k$ and $(\ln y_1)^k$
   are the major terms of the left and right hand side, respectively.
    This forces that $x_1=y_1$, and our claim is proved. Consequently, we have
 $\frac{a_1^*}{b_1^*}=\left ( \frac{N}{N'} \right )^{1/(1-\sigma)}.$
Furthermore, since the coefficients of the major terms must equal,  we get
$
\frac{a_1^*}{b_1^*}=\left ( \frac{M_1'}{M_1} \right )^{1/\sigma}.
$
Subtracting the  term involving $x_1$ and $y_1$ in \eqref{main_equ},
 and repeating  the above argument,
we have $\tilde p=\tilde q$, $x_j=y_j$ for $j=2,\dots, \tilde p$,
  and   the coefficients of the terms involving $x_j$
and $y_j$ coincide.
 Summing up the above discussion, we obtain
\begin{equation}\label{sec2-eq4}
\frac{a_j^*}{b_j^*}=\left ( \frac{N}{N'} \right )^{1/(1-\sigma)}
 =\left ( \frac{M_j'}{M_j} \right )^{1/\sigma}
\end{equation}
for all $j=1,\dots,\tilde p$.
It follows that
$$
\frac{s'}{s}=\frac{\sum_{j=1}^{\tilde p} M'_j}{\sum_{j=1}^{\tilde p} M_j}=\left ( \frac{N}{N'} \right )^{\sigma/(1-\sigma)},
$$
which together with \eqref{sec2-eq4} imply  \eqref{eq:invariants}.

For implication of the other direction, it is easy to show that  \eqref{eq:invariants} implies \eqref{equ_EF},
which means
$\beta_E(t)=\beta_F(t)$ for all $t>0$.
Therefore, $\mu_E$ and $\mu_F$ have the same multifractal spectrum by Theorem \ref{thm:JR}.
 The theorem is proved.
\end{proof}

\section{\bf{Approximate squares of  self-affine carpets}}\label{sec:vacant}

Let $E=K(n,m,\SD)$ be a self-affine carpet.
Let $\bx=x_1\dots x_k\in \{0,1,\dots, n-1\}^k$ and $\by=y_1\dots y_{\ell(k)}\in \{0,1,\dots, m-1\}^{\ell(k)}$.
Recall that
$
{Q}(\bx,\by)=({0.\bx}|_n, {0.\by}|_m)+
 \left [0,\frac{1}{n^k}\right ]\times \left [0,\frac{1}{m^{\ell(k)}}\right ]
$
 is  an  approximate square  of rank $k$, if $(x_j, y_j)\in \SD$ for $j\leq k$ and
$y_j\in \SE$ for $j>k$ (see Section \ref{sec:geo}).

Let $Q$ and $Q'$ be two approximate squares. We say  $Q'$ is  an \emph{offspring} of $Q$ if $Q'\subset Q$, and it is called a
\emph{direct offspring}  of $Q$
if the rank of $Q'$ equals the rank of $Q$ plus $1$. We use $\bx\ast \by$ to denote the concatenation of two words.
The following lemma is obvious, see for instance \cite{M84, King95}.

 \begin{lemma}\label{lem:offspring}  Let $E=K(n,m,\SD)$ be a self-affine carpet. Let $Q(\bx, \by)$ be an approximate square   of rank $k$ of $E$. Then

 (i) if $\ell(k)>k$, then the direct offsprings of $Q(\bx, \by)$ are
 $$
\left \{Q(\bx \ast u, \by \ast \bz);~(u, y_{k+1})\in \SD \text{ and } \bz\in \SE^{\ell(k+1)-\ell(k)}\right \},
$$
and $Q(\bx, \by)$ has $a_{y_{k+1}}\cdot s^{\ell(k+1)-\ell(k)}$ direct offsprings.

(ii)  if $\ell(k)=k$, then the direct offsprings of $Q(\bx, \by)$ are
$$
\left \{Q(\bx \ast u, \by \ast v \ast \bz);~(u, v)\in \SD \text{ and } \bz\in \SE^{\ell(k+1)-(k+1)}\right \},
$$
 and $Q(\bx, \by)$ has $Ns^{\ell(k+1)-(k+1)}$ direct offsprings.
\end{lemma}

Now we give some notations.
Let ${\mathbf E}_k $ be the union of all approximate squares of rank $k$.
Let $U$ be a connected component of $\mathbf E_k$;
 hereafter, we will  call $U$  a \emph{component of $\mathbf E_k$} for simplicity. An approximate square of rank $k$ contained in $U$ will be called a \emph{member} of $U$.
  Denote by $\#_k(U)$ the number of members of $U$.

 In $\{0,1,\dots, m-1\}^k$, we set $\prec$ to be the lexicographical order;   we denote
 by $(c_1\dots c_k)^+$ the word larger than and adjacent  to $c_1\dots c_k$.
 For $j\in \SE$, we define $\varphi_j(y)=(y+j)/m$. Denote $\pi(x,y)=y$.

  We shall show that if $E$ possesses vacant rows, then $\#_k(U)$ has a uniform upper bound.

\begin{lemma}\label{lem:cartesian}
Let $E=K(n,m ,{\cal D})\in {\mathcal M}_{t,v}(n,m)$.
Let  $U$ be a component of $\mathbf E_k$. Then
 $$\#_k(U)\leq (m-1)  M_0:=L_0$$
 where $ M_0$ is the constant in Theorem \ref{thm:finite-type}.
 Moreover, there exists $\by\in \SE^{\ell(k)-1}$
such that either $\pi(U)\subset \varphi_{\by}([0,1])$ or
$\pi(U)\subset \varphi_{\by}([0,1])\cup \varphi_{\by^+}([0,1])$.
\end{lemma}

\begin{proof}
Since $E$ contains vacant row, then $\pi(U)$ is contained either in
$\varphi_{\by}([0,1])$ or
in $\varphi_{\by}([0,1])\cup \varphi_{\by^+}([0,1)$
for some $\by \in \SE^{\ell(k)-1}$, which confirms the second assertion.

Let $Q$ be the connected component of $\widetilde{\mathbf E}_k$ containing $U$. Then $Q$ contains at most
$M_0$ number of basic rectangles of rank $k$ by Theorem \ref{thm:finite-type}. Since $E$ contains vacant row,
then every basic rectangle in $Q$ contribute at most $m-1$ approximate squares to $U$,
so $\#_k(U)\leq (m-1)  M_0$.
The lemma is proved.
\end{proof}

\begin{remark}\label{lem:disconnected} {\rm  \textbf{(A Criterion for totally disconnectedness.)}
Let $E=K(n,m,{\mathcal D})$ and ${\mathcal D}$ possess vacant rows. If $E$ has a non-trivial connected component, then the component must be a horizontal line segment.  Hence, $E$ is totally disconnected if and only if $ a_j<n$ for all $0\leq j\leq m-1$.
}
\end{remark}

For a self-affine carpet $F$, we use $\mathbf {F}_{k}$ to denote the union of all approximate squares of rank $k$ of $F$.
The following is a second variation of   \cite[Lemma 3.2]{FaMa92}.

\begin{lemma}\label{lem:Falconer+}
 Let $E, F\in {\cal M}_{t,v}(n,m)$.
Let $f:E\to F$ be a bi-Lipschitz map. Then there exists  integer $p_0$
such that,
 for any $k\geq 1$ and any component  $U$  of $\mathbf {E}_{k}$,
  there exist a group of components  of $\mathbf {F}_{k+p_0}$, which we denote by $J_j$, $1\leq j\leq q$,
such that
\begin{equation}\label{eq:Falconer}
f(U \cap E)=\bigcup_{j=1}^q (J_j\cap F),
\end{equation}
and all $J_j$ are offsprings of a component $I$ of $\mathbf F_{k-p_0}$.
\end{lemma}

\begin{proof} Let $C_0$ be a Lipschitz constant of $f$.
Set $p_0=\lfloor\log_n (2mC_0L_0)\rfloor+1$ where $L_0$ is the constant in Lemma \ref{lem:cartesian}.
Let $J$ be a  connected component of $\mathbf F_{k+p_0}$. We claim that
\begin{equation}\label{relation_EF}
\text{either}~~(J\cap F)\subset f(U\cap E), ~~~\text{or}~~(J\cap F)\cap(f(U\cap E))=\emptyset.
\end{equation}
The proof of   the claim
is exactly the same as the proof of Lemma \ref{lem:Falconer} and we omit it.
Clearly \eqref{relation_EF} implies \eqref{eq:Falconer}.

Now we prove that all $J_j$ are offsprings of a component $I$ of $\mathbf F_{k-p_0}$.
Applying the above claim  to the map  $f^{-1}:~~ F\rightarrow E$,
 we obtain that  for any component $W\in \mathbf{E}_k$   and any component $V\in \mathbf{F}_{k-p_0},$
 it holds that
$$
\text{either}  ~~~W\cap E\subset f^{-1}(V\cap F)~~~\text{or}~~~(W\cap E)\cap f^{-1}(V\cap F)=\emptyset.
$$
Set $W=U$ and let $I$ be the component of $\mathbf{F}_{k-p_0}$ such that
$U\cap E\subset f^{-1}(I\cap F),$ then
$
\bigcup_{j=1}^q (J_j\cap F)=f(U\cap E)\subset I\cap F,
$
the lemma is proved.
\end{proof}

\section{\textbf{Arithmetic doubling property}}\label{sec:arithmetic}

In this section, we  show that if  $\mu_E$ is doubling, then it is also doubling in an arithmetic sense.

\textbf{Notations about words.}
We use $\varepsilon_0$ to denote the empty word.
Let $S$ be the shift operator on words defined by ${ S}(c_1\dots c_k)=c_2\dots c_k$. Set $\chi^q(x_1\dots x_k)=x_1\dots x_q$  to be the
 prefix of $x_1\dots x_k$ with length $q$; especially $\chi(x_1\dots x_k)=x_1$.

  For a word $\bc=c_1\dots c_k$ over integers, we denote  $\prod \bc=\prod_{j=1}^kc_j;$ we make the convention that the value of the empty word $\varepsilon_0$ is $1$.

\textbf{Functions related to the distribution sequence.} For $j\in\{0,1,\dots, m-1\}$, we denote  $a(j)=a_j$;
 moreover, for $y_1\dots y_k\in \{0,1,\dots, m-1\}^k$, we define $a(y_1\dots y_k)$ to be
 the word $a(y_1)\dots a(y_k).$
Denote
$$
{\cal A}=\{a_j;~j\in {\mathcal E}\}=\{a_1^*,\dots, a_{\tilde p}^*\}.
$$
Recall that $M_j$ is the occurrence of $a_j^*$ in the distribution sequence. We define
$M: {\cal A}\to \{M_1,\dots, M_{\tilde p}\}$ by
$M(a_j^*)=M_j$; moreover, if $c_1\dots c_k\in  {\cal A}^k$,
we define
$M(c_1\dots c_k)=\prod_{j=1}^k M(c_j).$

\begin{defi} {\rm Let $W=Q(\bx,  \by)$ be an approximate square of rank $k$, where $\by=y_1\dots y_{\ell(k)}$. We define its \emph{color} to be the word $a(y_{k+1}y_{k+2}\dots y_{\ell(k)})$ over ${\cal A}$ if $\ell(k)>k$, and to be the empty word $\varepsilon_0$ if $\ell(k)=k$.
 }
\end{defi}

The next lemma counts the number of offsprings of an approximate square.

\begin{lemma}\label{lem:offspring+} Let $E\in {\mathcal M}_{t,v}(n,m)$, and let $W=Q(\bx, \by)$ be an approximate square of rank $k$ with color $\bc$.

(i) $\mu_E(W)= \prod \bc/N^{\ell(k)} .$

(ii) If $\ell(k)=k$, the set of colors of  direct offsprings of $W$
is $ {\cal A}^{\ell(k+1)-k-1},$
 and for any $\bw'$ in the above set, the number of  direct offsprings with  color $\bw'$ is $N  M(\bw')$.

(iii) If $\ell(k)>k$,   the set of colors of  direct offsprings of $W$  is
$$
\{S(\bc\ast\bz);~ \bz\in {\cal A}^{\ell(k+1)-\ell(k)}\},
$$
and for each $\bw'=S(\bc\ast\bz)$,   the
 number of   direct offsprings with this color is $\chi(\bc)M(\bz)$.
\end{lemma}

 \begin{proof} (i) This is proved in \cite{M84,King95} for general self-affine measures.

 (ii)
 Let $W'=Q(\bx\ast x_{k+1},\by\ast y_{k+1}\dots y_{\ell(k+1)})$ be a direct offspring of
 $Q(\bx,\by)$,  we have
 \begin{equation}\label{color-0}
 (x_{k+1},y_{k+1})\in\SD ~~ \text{and}~~ y_j\in \SE \text{ for }j=k+2,\dots, \ell(k+1).
 \end{equation}
 The color of $W'$ is $a(y_{k+2}\dots y_{\ell(k+1)})\in {\cal A}^{\ell(k+1)-k-1}$.
This proves the first assertion of (ii).
Once a color $\bw'$ is fixed, we have $N$ choices of $(x_{k+1},y_{k+1})$, and we have
$M(a(y_j))$ choice for $y_j$, $k+2\leq j\leq \ell(k+1)$, so
$W$ has
$N  M(a(y_{k+2}\cdots y_{\ell(k+1)}))=N   M(\bw') $ number of direct offsprings with the color $\bw'$.

(iii) Let $W'=Q(\bx*x_{k+1},\by*y_{\ell(k)+1}\cdots y_{\ell(k+1)})$ be a direct offspring of $Q(\bx,\by)$, then
\begin{equation}\label{color-1}
(x_{k+1},y_{k+1})\in\SD ~~\text{and}~~ y_j\in \SE \text{ for } j=\ell(k)+1, \dots, \ell(k+1).
\end{equation}
Denote
$\bz=a(y_{\ell(k)+1}\dots y_{\ell(k+1)})\in{\cal A}^{\ell(k+1)-\ell(k)}.$
Since $\bc=a(y_{k+1}\cdots y_{\ell(k)})$,  the color of $W'$ is
$$a(y_{k+2}\cdots y_{\ell(k)}y_{\ell(k)+1}\dots y_{\ell(k+1)})=S(\bc\ast\bz).$$
On the other hand,  fix a color  $\bw'=S(\bc\ast\bz)$, the choices of $x_{k+1}$
is $a(y_{k+1})=\chi(\bc)$, and for each $j\geq \ell(k)+1 $, the choices of $y_j$ is
$M(a(y_j))$, so the total number of choices is
$$
a(y_{k+1})M(a(y_{\ell(k)+1}\dots y_{\ell(k+1)}))=\chi(\bc) M(\bz).
$$
The lemma is proved.
\end{proof}

The next lemma says if $\mu_E$ is doubling, then the approximate squares  in a component of $\mathbf E_k$  have `almost' the same   color.

\begin{lemma}\label{lem:differ} Let $E\in {\mathcal M}_{t,v,d}(n,m)$, and let $U$ be a component of $\mathbf E_k$.
Then the colors of two members of $U$  differ  at most at two entries.
\end{lemma}

\begin{proof} Pick  $Q(\bx,\by),Q(\bx',\by')\in U$.
By Remark \ref{rem:Wen}, $\mu_E$ is doubling if and only if at least one of the following
condition holds: (i) $a_ja_{j+1}=0$ holds for $0\leq j\leq m-2$; (ii) $a_0a_{m-1}=0$;  (iii) $a_0=a_{m-1}$.

In case of (i), $Q(\bx,\by)$ and $Q(\bx',\by')$ must be located in the same row, which implies that
$\by=\by'$.

In case of (ii), we have $\pi(U)\subset \varphi_{\chi^{\ell(k)-1}(\by)}([0,1])$, so $\chi^{\ell(k)-1}(\by)=\chi^{\ell(k)-1}(\by')$,
then the colors of $Q(\bx,\by),Q(\bx',\by')$ can only differ at the last entries.

In case of (iii),  by Lemma \ref{lem:cartesian}, there exists
a word $\by^*\in \SE^{\ell(k)-1}$, such that $\pi(U)$  falls into the following two cases.

\emph{Case (i).}  $\pi(U)\subset \varphi_{\by^*}([0,1])$.

In this case, we have $\chi^{\ell(k)-1}(\by)=\chi^{\ell(k)-1}(\by')=\by^*$.

\emph{Case (ii).}  $\pi(U)\subset \varphi_{\by^*}([0,1])\cup \varphi_{(\by^*)^+}([0,1])$.

Then
there exists $h\geq 1$ such that $y_{h}, y_{h}+1\in \SE$, and
$$
\by^*=y_1\dots y_{h-1} y_h   (m-1)^{\ell(k)-h-1}  , \quad
(\by^*)^+ =y_1\dots y_{h-1}(y_{h}+1)0^{\ell(k)-h-1} .
$$
Since $a(0)=a(m-1)$, we see that the colors of $Q(\bx,\by)$ and $Q(\bx',\by')$ are digit-wisely equal except at the positions
 $h$ and $\ell(k)$.
\end{proof}

The following lemma shows that measures of approximate squares in a component of $\mathbf E_k$  change slowly in an arithmetic sense.

\begin{lemma}\label{lem:UU'} Let $E\in {\mathcal M}_{t,v,d}(n,m)$ and $U$ be a component of $\mathbf E_k$. Let $L_0$ be the constant in Lemma \ref{lem:cartesian} and denote $\mu=\mu_E$. Then

 (i) Let $B$ be a member of $U$, then
 $$
 \frac{\mu(U)}{\mu(B)}\in  \frac{\Z\cap (0,L_0n^{2m}) }{(a_1^*\cdots a_{\tilde p}^*)^2}.
 $$

 (ii) If    $U'$  is a direct offspring of $U$, then  there exists a positive integer
 $H$ (independent of the choice of $k$ and $U$) such that
 $$
 \frac{\mu(U')}{\mu(U)}\in \frac{\Z}{H}.
 $$
\end{lemma}

\begin{proof} First, we show that  if $B_1$ and $B_2$ are two members of $U$, then
 \begin{equation}\label{eq:BB}
 \frac{\mu(B_1)}{\mu(B_2)}\in \frac{\Z\cap (0, n^{2m})}{(a_1^*\cdots a_{\tilde p}^*)^2}.
 \end{equation}
 Let $\bc_1$ and $\bc_2$ be the colors of $B_1$ and $B_2$, respectively. Then
by Lemma \ref{lem:offspring+} (i) and Lemma \ref{lem:differ},
$$
\frac{\mu(B_1)}{\mu(B_2)}=\frac{\prod \bc_1}{\prod \bc_2}=\frac{\bc_1(i)\bc_1(j)}{\bc_2(i)\bc_2(j)}.
$$
 Notice that $\frac{\bc_1(i)\bc_1(j)}{\bc_2(i)\bc_2(j)}\cdot(a_1^*\cdots a_{\tilde p}^*)^2$
is an integer,  and it is less than $n^{2m}$ because $a_j^*<n$ and $\tilde p<m$. This proves \eqref{eq:BB}.

(i)  Let $B_j$, $j=1,\dots, h$ be the members of $U$. By \eqref{eq:BB},
$$
\frac{\mu(U)}{\mu(B)}=\frac{\sum_{j=1}^h \mu(B_j)}{\mu(B)}=\frac{\sum_{j=1}^h n_j}{(a_1^*\cdots a_{\tilde p}^*)^2},
$$
where $n_j$ are integers in $(0, n^{2m})$. Meanwhile, $h\leq L_0$ by Lemma \ref{lem:cartesian}, so $\sum_{j=1}^h n_j< L_0 n^{2m}$.  This proves (i).

(ii) Let $B'$ be a member of  $U'$ and let $B$ be its direct ancestor. Clearly $B$ is a member
of $U$.
Let $\bw$ and $\bw'$ be the colors of $B$ and $B'$, respectively.
If $\ell(k)>k$, by Lemma \ref{lem:offspring+},
 $\bw'=S(\bw \ast \bu)$ for some $\bu\in {\mathcal A}^{\ell(k+1)-\ell(k)}$, and  it follows that
\begin{equation}\label{eq:BB'}
\frac{\mu(B')}{\mu(B)}=\frac{\prod S(\bw \ast \bu)/ N^{\ell(k+1)} }{\prod \bw /N^{\ell(k)}}
=\frac{\prod  \bu }{\chi(\bw) N^{\ell(k+1)-\ell(k)}} \in \frac{\Z}{(a_1^*\cdots a_{\tilde p}^*)N^{\lfloor1/\sigma \rfloor+1  }}.
\end{equation}
If $\ell(k)=k$, we have $\bw=\varepsilon_0$ and $\bw'\in {\cal A}^{\ell(k+1)-k-1}$, and
it is easy to see that the inclusion relation of
\eqref{eq:BB'} still holds.

By item (i) we just proved, we have
$$
\frac{\mu(B)}{\mu(U)}\in \frac{\Z}{(L_0n^{2m})!}.
$$
Set $H=(L_0n^{2m})!\cdot {(a_1^*\cdots a_{\tilde p}^*)^3 N^{\lfloor1/\sigma \rfloor+1  }}$,  we obtain (ii).
\end{proof}

\section{\textbf{Measure preserving property  }}\label{sec:preserve}

In this section, we prove the measure preserving property of
bi-Lipschitz maps between sets in ${\cal M}_{t,v,d}(n,m)$.
For  a self-affine carpet $E$, we  denote by ${\mathcal C}_{E,k}$ the collection of components of $\mathbf E_k$, and set ${\mathcal C}_E=\bigcup_{k\geq 0} {\mathcal C}_{E,k}$, where we set ${\mathcal C}_0=\{[0,1]^2\}$ by convention.

Suppose $f:~E\to F$ is a bi-Lipschitz map with Lipschitz constant $C_0$.
Define
$$\tau(U)=\frac{\mu_F(f(U\cap E))}{\mu_E(U\cap E)}, \quad \text{ where } U\in {\mathcal C}_E.$$
Since  the measures $\mu_F\circ f$ and $\mu_E$ are equivalent (Theorem \ref{thm:equivalent}), we have
$$
\lambda=\sup_{U\in {\mathcal C}_E} \tau(U)<\infty.
$$

\begin{proof}[\textbf{Proof of Theorem \ref{thm:measure} }] Recall that $E_{\bi}=S_{\bi}(E)$ for $\bi\in{\cal D}^k$.
Denote ${\cal B}=\{b_1^*,\dots, b_{\tilde p}^*\}$.
Since any $U\in {\mathcal C}_E$ contains a cylinder and vice versa,
the theorem holds if and only if
 there exists $U\in {\mathcal C}_E$ such that $f:~(U\cap E,\mu_E)\to (f(U\cap E),\mu_F)$ is measure preserving.
Suppose on the contrary that the theorem is false. Then for any $U\in {\mathcal C}_E$, $f|_{U\cap E}$  is not measure preserving.

 First, we observe that $\tau(U)<\lambda$ for any $U\in {\mathcal C}_E$; for otherwise,   $f|_{U\cap E}$ is measure preserving by the maximality of $\lambda$.

Let $H$ be the constant in Lemma \ref{lem:UU'}(ii) and  $p_0$  the constant in Lemma \ref{lem:Falconer+}.
Set
$$\epsilon=\frac{1}{2(1+H^{2p_0+2})}.$$
 By the definition of $\lambda$,  there exists  $U\in {\mathcal C}_E$  such that
$$\lambda(1-\epsilon)<\tau(U)<\lambda.$$
Let $k$ be the rank of $U$.
Let $U_1,\dots, U_p$ be the direct offsprings of $U$. Then either $\tau(U_j)=\tau(U)$ for all $1\leq j\leq p$,
 or there exists $1\leq h\leq p$ such that
 $\tau(U_h)> \tau(U).$
We may assume without loss of generality that the second scenario occurs, since  $f|_{U\cap E}$ is not measure preserving, then there exists a sequence $U=V_0, V_1,\dots, V_r$
 in ${\cal C}_E$ such that $V_i$ is a direct offspring
 of $V_{i-1}$ for $i=1,\dots, r$, and $\tau(V_0)=\cdots=\tau(V_{r-1})\neq \tau(V_r)$.
 Without loss of generality, we may assume that $\tau(V_{r-1})< \tau(V_r)$.
 So we can replace $U$ by $V_{r-1}$ to start our discussion.
In the following,  we estimate
 $\displaystyle \tau(U_h)={\mu_F(f(U_h\cap E))}/{\mu_E(U_h)}.$
By Lemma \ref{lem:Falconer+}, $f(U\cap E)$ can be decomposed into
$$f(U \cap E)=\bigcup_{j=1}^q (J_j\cap F),$$
where $J_j$ are components of  $\mathbf F_{k+p_0}$, and  are offsprings of a component $J^*\in \mathbf F_{k-p_0}$.
Applying
Lemma \ref{lem:Falconer+} to $U_h$, we see that there exists $I_1,\dots, I_t$, which are components of $\mathbf F_{k+1+p_0}$, such that $f(U_h\cap E)=\bigcup_{i=1}^t (I_i\cap F)$. (Obviously  $I_1,\dots, I_t$ are offsprings of $J^*$.)

Denote
$$\alpha={\mu_F(J^*)}/{H^{2p_0+1}}.$$
Notice that    $f(U\cap E)$ consists of  $2p_0$-step offsprings of $J^*$, and
$f(U_h\cap E)$  consists of $(2p_0+1)$-step offsprings of $J^*$, we obtain that
both of the measures of them are multiples of $\alpha$ by Lemma \ref{lem:UU'}(ii).
Denote
$$u=\mu_F(f(U\cap E))/\alpha \text{ \ and \ } u'=\mu_F(f(U_h\cap E))/\alpha.$$

Since $f(U\cap E)\subset J^*$, we have
$u\leq {\mu_F(J^*)}/{\alpha}=H^{2p_0+1}.$
Moreover,
\begin{equation}\label{eq:Lu}
\tau(U_h) =\frac{u' \alpha}{\mu_E(U_h)}= \frac{u'}{u}\cdot \frac{\mu_E(U)}{\mu_E(U_h)}\cdot \tau(U).
\end{equation}
 By Lemma \ref{lem:UU'} and ${\mu_E(U_h)}/{\mu_E(U)}\leq 1$,
   the denominator of ${\mu_E(U)}/{\mu_E(U_h)}$ is smaller than $H$.
Since $\tau(U_h)>\tau(U)$,  we have
$$
\tau(U_h)-\tau(U)\geq  \frac{1}{Hu}\cdot \tau(U)\geq \frac{1}{H^{2p_0+2}} \tau(U).
$$
 It follows that
$
\tau(U_h) \geq \tau(U)\left (1+\frac{1}{H^{2p_0+2}}\right )>\lambda,
$
which is a contradiction. The theorem is proved.
\end{proof}

\section{\textbf{Proof of Theorem \ref{thm:irrational}}}\label{sec:irrational}
In this section, we  prove Theorem \ref{thm:irrational} by a number theoretic argument.

 Let $p$ be a prime number.  The  $p$-adic valuation function $v_p(k)$  denotes the number of factor $p$ contained in $k\in \mathbb{Z}$.
For a rational number $k_1/k_2$, set  $v_p(k_1/k_2)=v_p(k_1)-v_p(k_2)$. For   $x\in \Q$, we define
$|x|_p=p^{-v_p(x)}$, which is a non-archimedean absolute value on $\Q$. See \cite{Fern97}.

From now on, we assume that $E, F\in {\mathcal M}_t(n,m)$  and  $\mu_E$ and $\mu_F$ have the same multifractal spectrum.  By Theorem \ref{thm:invariants}, we have
\begin{equation}\label{equ_a_ib_i}
\frac{a_i^*}{b_i^*}= \left(\frac{N}{N'}\right)^{1/(1-\sigma)}, \ \
~~\text{for }~~ i=1,\dots,\tilde p.
\end{equation}

Denote
${\cal A}=\{a_1^*,\dots ,a_{\tilde p}^*\}$ and ${\cal B}=\{b_1^*,\dots ,b_{\tilde p}^*\}.$
For $1\leq j \leq \tilde p$, we call $b_j^*$ the \emph{dual} of $a_j^*$ and vice versa. Moreover, we  say a word $\bz=z_1\dots z_k\in {\cal B}^k$ is the dual of $\bw=w_1\dots w_k\in {\cal A}^k$, if
$z_j$ is the dual of $w_j$ for all $j=$$1,\dots, k$.

Let $W$ and $W'$ be  approximate squares of $\mathbf{E}_k$ and $\mathbf{F}_k$ with color $\bc$ and $\bc'$, respectively.
If $\bc'$ is the dual of $\bc$, by Lemma \ref{lem:offspring+} and \eqref{equ_a_ib_i},  we have
\begin{equation}\label{eq:gammak}
\frac{\mu_E(W)}{\mu_F(W')} = \frac{\prod \bc \cdot (N')^{\ell(k)}}{ \prod \bc' \cdot N^{\ell(k)}}=
\left(\frac{a_1^*}{b_1^*}\right)^{\ell(k)-k}\left( \frac{N'}{N}\right)^{\ell(k)}:=\gamma_k.
\end{equation}

 \begin{lemma}\label{dis_gamma}
 If $\sigma\in \Q^c$ and $N\neq N'$, then

 (i) $\gamma_k<m$ and all $\gamma_k$'s are distinct rational numbers.

 (ii) There exists  a prime factor $p$ of $b_1^*N$ such that $v_p(\gamma_k)$ tends to $-\infty$ as $k\to \infty$.
 \end{lemma}

 \begin{proof} (i) Denote by $\{x\}$ the fractional part of $x$. By \eqref{equ_a_ib_i} with $i=1$, we have
 $$\gamma_k=\left ({a_1^*}/{b_1^*}\right )^{-\sigma \{{k}/{\sigma}\}}<n^\sigma=m.$$
 Since $\sigma\in \Q^c$, we have $\{k/\sigma\}$ are distinct,  so $\gamma_k$ are distinct. Item (i) is proved.

 (ii) Let $p$ be a prime factor of  $b_1^*N$ such that $u:=v_p(a_1^*/b_1^*)$ and $u':=v_p(N'/N)$ are not simultaneously $0$.
 Then $u(1/\sigma-1)+u'/\sigma \neq 0$ by the irrationality of $\sigma$. So
 $$v_p(\gamma_k)=u(\ell(k)-k)+u' \ell(k)=k(u({1}/{\sigma}-1)+{u'}/{\sigma})-(u+u')\{{k}/{\sigma}\} $$
 either tends to $+\infty$ or tends to $-\infty$ as $k\to \infty$.
 Since $\gamma_k<m$, we have that $v_p(\gamma_k)$ tends to $-\infty$ for at least one prime factor of $b_1^*N$.
 \end{proof}

\begin{proof}[\textbf{Proof of Theorem \ref{thm:irrational}.}]
By Corollary \ref{cor:spectra}, the assumption $E\sim F$  implies that $\mu_E$ and $\mu_F$
have the same multifractal spectrum. Hence, by Theorem \ref{thm:invariants},
  $(a_j)_{j=0}^{m-1}$
is a permutation of $(b_j)_{j=0}^{m-1}$ if and only if $N=N'$.
Suppose on the contrary that $N\neq N'$.

Let $p$ be a prime factor of $b_1^*N$ such that   $v_p(\gamma_k)$ tends to $-\infty$ as $k\to \infty$ by Lemma \ref{dis_gamma}(ii).
Let $h$ be an  index such that $v_p(a_h^*)=\min \{v_p(a_j^*); ~j\in \{1,\dots, \tilde p\}\}$.
Let $d_0=(x_0,y_0)$ be an element of $\SD$ satisfying  $a({y_0})=a_h^*$.

Let $U_0$ be an element in ${\mathcal C}_E$ such that $f|_{U_0}$ is measure preserving  (see Theorem \ref{thm:measure}).
Denote the rank of $U_0$ by $k_0$. Let $E_{\bi}$ be a cylinder of rank $\ell(k_0)$ contained in $U_0$, and  let $z_0\in E_{\bi}$ be the point with coding $\bi(d_0)^\infty$.

Pick  $k>\ell(k_0)$. Let $U=U_k$ be the  offspring  of $U_0$  with rank $k$ such that $z_0\in U$.
Let $B$ be a member of $U$ containing $z_0$, then $B$ has color  $\ba^*= (a_h^*)^{\ell(k)-k}$.
By Lemma \ref{lem:Falconer+},
$f(U\cap E)=\bigcup_{j=1}^q (J_j\cap F),$
where $J_j$ are components of ${\mathbf F}_{k+p_0}$.
By Lemma \ref{lem:UU'} (i), there exists an positive integer   $u<L_0n^{2m}$ such that
$$
\mu_E(U)=\frac{u }{(a_1^*\dots a_{\tilde p}^*)^2}\mu_E(B)=
\frac{u\prod \ba^*}{N^{\ell(k)}(a_1^*\dots a_{\tilde p}^*)^2}.
$$
  Similarly, for $j=1,\dots, q$,  let $\bb_j$ be
the color of left-bottom member  of $J_j$,  then we have
$$
\mu_F(J_j)=\frac{u_j\prod \bb_j}{(N')^{\ell(k+p_0)}(b_1^*\dots b_{\tilde p}^*)^2}
$$
for some integer $u_j< L_0n^{2m}$, where $L_0$ is a constant such that Lemma \ref{lem:cartesian} holds for $E$
and $F$ simultaneously.

In the following  we estimate  $|\mu_F(J_j)/\mu_E(U)|_p$.

Let $\ba_j$ be the dual of $\bb_j$, and let $n_0=|\bb_j|-|\ba^*|=\ell(k+p_0)-\ell(k)-p_0$.
 By (\ref{eq:gammak}), we have
\begin{eqnarray*}\label{eq:muEF}
\frac{\mu_F(J_j)}{\mu_E(U)} &=&\frac{u_j\prod \bb_j \cdot N^{\ell(k)} } {u\prod \ba^* \cdot {(N')^{\ell(k+p_0)} } }
\cdot \left(\frac{a_1^*}{b_1^*}  \right)^{2\tilde{p}}
 = \frac{u_j\prod_{i=1}^{n_0} \bb_j(i)}{u\cdot (N')^{n_0+p_0}} \left(\frac{a_1^*}{b_1^*}  \right)^{2\tilde{p}}\cdot \frac{\prod { S}^{n_0}(\ba_j)}{\prod \ba^*}
\cdot  \gamma_k^{-1}\\
&:=&H_1\cdot H_2\cdot \gamma_k^{-1}.
\end{eqnarray*}
The numerator and denominator of $H_1$ are both bounded, so $|H_1|_p$ is also bounded.
As for $H_2$, we have $ |H_2|_p \leq 1$ by the minimality of $a_h^*$.
Therefore,
$$
\left |\frac{\mu_F(J_j)}{\mu_E(U)}\right |_p=|H_1|_p\cdot |H_2|_p \cdot |\gamma_k^{-1}|_p\to 0 \quad \text{ as }k\to\infty.
$$

Denote $\lambda=\mu_F(f(E_{\bi}))/\mu_E(E_{\bi})$, which is apparently   a rational number, since $f(E_{\bi})$ is a finite union of cylinders of $F$.
On one hand, by the measure preserving property, we have  for all $k>\ell(k_0)$,
$$\left |\frac{\mu_F\circ f(U\cap E)}{\mu_E(U)}\right |_p=|\lambda|_p;$$
on the other hand, we have
$$
\left |\frac{\mu_F\circ f(U\cap E)}{\mu_E(U)}\right |_p
=\left |\frac{\sum_{j=1}^q \mu_F(J_j)}{\mu_E(U)}\right |_p
\leq \max_{j=1,\dots, q} \left |\frac{\mu_F(J_j)}{\mu_E(U)}\right |_p\to 0
$$
as $k\to \infty$. This is a contradiction, and the theorem is proved.
\end{proof}



\begin{thebibliography}{99}


\bibitem{Bar} J. Barral and M. Mensi, {\it Gibbbs measures on self-affine Sierpi$\acute{n}$ski
             carpets and their singularity spectrum,} Ergod. Th. \& Dynam. Sys., \textbf{27}(5) (2007), 1419-1443.

\bibitem{Bed84}  T. Bedford, {\it Crinkly curves, Markov partitions and dimensions,} PhD Thesis, University of Warwick, 1984.

\bibitem{CP88} D. Cooper and T. Pignataro: {\it On the shape of Cantor sets,}
J. Differential Geom., \textbf{28} (1988), 203--211.

 \bibitem{DS}
G.~David and S.~Semmes, {\it Fractured fractals and broken dreams :
self-similar geometry through metric and measure}, Oxford Univ.\
Press,~1997.

  \bibitem{FaMa92} K. J.~Falconer and D. T.~Marsh, {\it On the Lipschitz equivalence of
Cantor sets}, Mathematika, \textbf{39} (1992), 223-233.


\bibitem{Fal90} K. J. Falconer,
{\it Fractal Geometry, Mathematical Foundations and Applications,} Wiley, New York, 1990.

\bibitem{Fern97} Q. G. Fernando, {\it $p$-adic Numbers, } Springer, Second Edition 1997.


\bibitem{Hut81} J. E. Hutchinson,  {\it Fractals and self-similarity,} Indiana Univ. Math. J., \textbf{30} (1981), 713-747.


\bibitem{JR11} T. Jordan and M. Rams, {\it  Multifractal analysis for Bedford-McMullen carpets,}
Math. Proc. Camb. Phil. Soc., \textbf{150} (2011), 147-156.

\bibitem{King95} J. F. King, \textit{The singularity spectrum for general Sierpinski carpets,} Adv. Math., \textbf{116} (1995), 1-11.

    \bibitem{Mattila} M. Llorente and P. Mattila, {\it Lipschitz equivalence of
    subsets of self-conformal sets.} Nonilearity, \textbf{23} (2010), 875--882.


\bibitem{Miao2013}  B. M. Li, W. X. Li and J. J. Miao, {\it Lipschitz equivalence of McMullen sets,} Fractals, \textbf{21}(3 \& 4) (2013), 1350022, 11 pages.

\bibitem{LWW16} H. Li, C. Wei and S. Y. Wen, {\it Doubling property of self-affine measures on carpets of Bedford and McMullen,} Indiana Univ. Math. J., \textbf{65} (2016), 833-865.

 \bibitem{LL13} J. J. Luo and K. S. Lau, \textit{Lipschitz equivalence of self-similar sets and
 hyperbolic boundaries,} Adv. Math., \textbf{235} (2013), 555-579.


\bibitem{MM11}  J. M. Mackay, {\it Assouad dimension of self-affine carpets.}(English summary) Conform. Geom. Dyn., \textbf{15}(2011), 177-187.

\bibitem{M84} C. McMullen, \textit{The Hausdorff dimension of general Sierpi$\acute{n}$ski carpets,} Nagoya Math. J., \textbf{96} (1984), 1-9.

\bibitem{Miao2017} J. J. Miao, L. F. Xi and Y. Xiong,
\textit{Gap sequences of McMullen sets,} Proc. Amer. Math. Soc., \textbf{145}(4)(2017), 1629-1637.

\bibitem{O98} L. Olsen, \textit{Self-affine multifractal Sierpi$\acute{n}$ski sponges in $\mathbb{R}^d$,} Pacific J. Math.,\textbf{183} (1998), no. 1, 143-199.

\bibitem{PY94} Y. Peres, \textit{The self-affine carpets of McMullen and Bedford have infinite Hausdorff measure,}
Math. Proc. Camb. Phil. Soc., \textbf{116}(1994), 513-526.

  \bibitem{RRW12}
H. Rao, H. J. Ruan and Y. Wang,  {\it Lipschitz equivalence of Cantor sets and algebraic properties of contraction ratios},  Trans. Amer. Math. Soc.,  \textbf{364} (2012), 1109-1126.



 \bibitem{RRX06}  H. Rao, H. J. Ruan and L. F. Xi, {\it Lipschitz equivalence of self-similar
 sets,} C. R. Math. Acad. Sci. Paris, \textbf{342}(2006), 191-196.

 \bibitem{RZ15} H. Rao and Y. Zhang,
\emph{Higher dimensional Frobenius problem and Lipschitz equivalence of Cantor sets,} \emph{J. Math. Pures Appl.}, \textbf{104} (2015), 868-881.


 \bibitem {XX10} L. F. Xi and Y. Xiong, \textit{Self-similar sets with initial cubic patterns}, C. R. Math. Acad. Sci. Paris, \textbf{348} (2010), 15-20.

\bibitem{XR08} L. F. Xi and H. J. Ruan, \textit{Lipschitz equivalence of self-similar sets satisfying strong separation condition}, Acta Math. Sin (Chin. Ser.), \textbf{51}(3)(2008), 493-500.

\bibitem{XiXi21} L. F. Xi and Y. Xiong,  \emph{Algebraic criteria for Lipschitz equivalence of dust-like self-similar sets}, J. Lond. Math. Soc., \textbf{103} (2021), 760-780.

\bibitem{YZ20a} Y. M. Yang and Y. Zhang, \textit{Lipschitz classification of  Bedford-McMullen carpets with uniform horizontal fibers,}
J. Math. Anal. Appl., \textbf{495} (2021), 12pp.

\bibitem{Zorich} V. Zorich, \textit{Mathematical Analysis}, Springer-Verlag, Berlin Heidelberg, 2004.
\end{thebibliography}
\end{document}